\documentclass[11pt]{amsart}
\usepackage{latexsym,amssymb,amsmath,youngtab}
\usepackage{youngtab}
\usepackage{epsfig}
\usepackage{amsmath,amsthm,amssymb,amscd,MnSymbol}
\usepackage[mathscr]{eucal}
\usepackage{verbatim, hyperref}
\usepackage{color}
\newcommand*{\textlabel}[2]{%
  \edef\@currentlabel{#1}
  \phantomsection
  #1\label{#2}
}

\newtheoremstyle{custom}
  {3pt}
  {3pt}
  {\slshape}
  {}
  {\bfseries}
  {.}
  { }
   {}
\theoremstyle{custom}
\newtheorem{theorem}{Theorem}[section]
\newtheorem{proposition}[theorem]{Proposition}

\newtheorem{proposition/definition}[theorem]{Proposition/Definition}
\newtheorem{lemma}[theorem]{Lemma}
\newtheorem{corollary}[theorem]{Corollary}

\theoremstyle{definition}
\newtheorem{definition}[theorem]{Definition}

\newtheorem{example}[theorem]{Example}

\theoremstyle{remark}
\newtheorem{remark}[theorem]{Remark}

\newtheoremstyle{exercise}
  {3pt}
  {6pt}
  {}
  {}
  {\bfseries}
  {:}
  { }
   {}
\theoremstyle{exercise}
\newtheorem{exercise}[theorem]{Exercise}
\newtheoremstyle{exercises}
  {3pt}
  {6pt}
  {}
  {}
  {\bfseries}
  {:}
  {\newline}
   {}

\theoremstyle{exercise}
\newtheorem{exercises}[theorem]{Exercises}

\input epsf
\def\boxit#1{\vbox{\hrule height1pt\hbox{\vrule width1pt\kern3pt
  \vbox{\kern3pt#1\kern3pt}\kern3pt\vrule width1pt}\hrule height1pt}}

\def\11{\mathbf 1}

\def\s{\sigma}

\def\ot{{\mathord{ \otimes } }}

\def\ra{{\mathord{\;\rightarrow\;}}}

\def\dim{{\rm dim}\;}

\def\s{\sigma}

\def\ra{\rightarrow}

\def\be{\begin{equation}}
\def\ene{\end{equation}}

\DeclareMathOperator{\tlog}{log}

\def\rank{\operatorname{rank}}

\DeclareMathOperator*{\perm}{perm}

\DeclareMathOperator*{\spa}{span}
\def\dim{{\rm dim }}
\def\NUM{{\rm NUM }}
\begin{document}

\title{Flattenings and Koszul Young flattenings arising in complexity theory }
\author{Yonghui Guan}
 \begin{abstract}
I find new equations for Chow varieties, their secant varieties, and an additional variety that arises in the study of complexity theory by flattenings and  Koszul Young flattenings. This enables a new lower bound for symmetric border rank of $x_1x_2\cdots x_d$ when $d$ is odd, and a new lower complexity bound for the permanent.
 \end{abstract}
 \email{yonghuig@math.tamu.edu }
\maketitle
\section{Introduction}
\subsection{Motivation from algebraic geometry}
There has been substantial recent interest in the equations of certain algebraic varieties that encode
natural properties of polynomials (see e.g., \cite{MR2310544,MR3169697,LMsec,MR3081636,LWsecseg}). Such varieties are usually preserved by algebraic
groups, this paper studies equations for several such varieties. One variety of interest
is the {\it Chow variety} of polynomials that decompose as a product of linear forms, which is defined by
$Ch_d(V)=\mathbb{P}\{z\in S^dV|z=w_1\cdots w_d\,{\rm\ for\ some\ } w_i \in V\}\subset\mathbb{P}S^dV,$
where $V$ is a finite-dimensional complex vector space and  $\mathbb{P}S^dV$ is the projective space of homogeneous polynomials of degree $d$ on the dual space $V^*$.

The ideal of the Chow variety of polynomials that decompose as a product of linear forms has been studied for over 100 years, dating back at least to  Gordon and Hadamard. Let $S^\delta(S^dV)$ denote the space of homogeneous polynomials of degree $\delta$ on $S^dV^*$. The {\it Foulkes-Howe} map $h_{\delta,d}:S^\delta(S^dV)\rightarrow S^d(S^\delta V)$ was defined by Hermite \cite{hermite} when $\dim\ V=2$, and Hermite proved the map is an isomorphism in his celebrated \lq\lq Hermite reciprocity\rq\rq. Hadamard \cite{MR1554881} defined the map in
general and observed that its kernel is $I_\delta(Ch_d(V^*))$, the degree $\delta$ component of the ideal
of the Chow variety. We do not understand this map when $d>4$ (see \cite{MR1243152,MR1601139,MR0037276,MR1504330,MR983608,MR1651092}).

Brill and Gordon (see \cite{MR2664658,gkz,Gordon,MR2865915}) wrote down set-theoretic equations
for the Chow variety of degree $d+1$,  called \lq\lq Brill's equations\rq\rq. Brill's equations give a geometric derivation of set-theoretic equations for the Chow variety. I computed Brill's equations in terms of a $GL(V)$-module from a representation-theoretic perspective in \cite{2015arXiv150802293G}, where $GL(V)$  denotes the {\it general linear group} of invertible linear maps from $V$ to $V$.

 Define the {\it Veronese variety}
$v_d(\mathbb{P}V)\subset \mathbb{P}S^dV$ by
\begin{eqnarray*}
v_d(\mathbb{P}V)=\mathbb{P}\{z\in S^dV|z=w^d\ {\rm for\ some}\ w \in V\}.
\end{eqnarray*}
Let $W$ be a complex vector space and let $X\subset \mathbb{P}{W^*}$ be an algebraic variety. Define
$\sigma^0_r(X)={\bigcup_{p_1,\dots,p_r\in X}\langle p_1,\dots,p_r\rangle} \subset \mathbb{P}W^*,$
where $\langle p_1,\dots,p_r\rangle$ denotes the  projective plane spanned by $p_1,\dots,p_r$.
Define the $r$-th {\it secant variety} of $X$ to be
$\sigma_r(X)=\overline{\sigma^0_r(X)} \subset \mathbb{P}W^*,$ where the overline denotes closure in the Zariski topology.

For a given polynomial $P\in S^dV$,
the {\it symmetric rank} $\mathbf{R}_S(P)$ of $P$ is the smallest $r$ such that $[P]\in \sigma^0_r(v_d(\mathbb{P}V))$,
the {\it symmetric border rank} $\underline{\mathbf{R}}_S(P)$ of $P$ is the smallest $r$ such that $[P]\in \sigma_r(v_d(\mathbb{P}V))$.
Notice that $\mathbf{R}_S(P)\geq\underline{\mathbf{R}}_S(P)$. It is an open question to determine the symmetric border rank of $x_1\cdots x _d\in S^dV$.

\subsection{Motivation from complexity theory}



Leslie Valiant \cite{vali:79-3} defined in 1979 an algebraic analogue of the ${\mathbf{P}}$ versus ${\mathbf{NP}}$ problem.
The class ${\mathbf{VP}}$ is an algebraic analogue of the class ${\mathbf{P}}$, and the class ${\mathbf{VNP}}$ is an algebraic analog of the class ${\mathbf{VP}}$. Valiant Conjectured $\mathbf{VP\neq VNP}$.
 Valiant's Conjecture $\mathbf{VP\neq VNP}$ \cite{vali:79-3} may be rephrased as  \lq\lq there does not exist polynomial size circuit that computes the permanent\rq\rq, defined by ${\perm}_n=\sum_{\sigma\in\mathfrak{S_n}}x_{1\sigma(1)}x_{2\sigma(2)}\cdots x_{n\sigma(n)}\in S^n{\mathbb{C}^{n^2}}$, where $\mathfrak{S_n}$ is the symmetric group and $\mathbb{C}^{n^2}$ has a basis $\{x_{ij}\}_{1\leq i,j\leq n}$.

A geometric method to approach Valiant's conjecture implicitly proposed by Gupta, Kamath, Kayal and Saptharishi \cite{DBLP:journals/eccc/GuptaKKS13} is to determine equations for the secant varieties defined in Theorems \ref{chowvnp} and \ref {con} below.

Let $h_n$ and $g_n$ be two positive sequences, define $h_n=\omega(g_n)$ if $\lim_{n\rightarrow\infty}\frac{h_n}{g_n}=\infty$,
define $h_n=\Omega(g_n)$ if $\lim_{n\rightarrow\infty}\frac{h_n}{g_n}\geq C$ for some positive constant $C$.
Define the {\it padded permanent} to be $l^{n-m}{\perm}_m \in S^n\mathbb{C}^{m^2+1}$, where $\mathbb{C}^{m^2+1}$ has a
basis $\{l, x_{ij}\}_{1\leq i,j\leq m}$.
The following two theorems appeared in \cite{MR3343444}, they are
geometric rephrasings of results in \cite{DBLP:journals/eccc/GuptaKKS13}.
\begin{theorem}{\rm \cite{DBLP:journals/eccc/GuptaKKS13,MR3343444}}\label{chowvnp}  If for all but a finite number of $m$, for all $r,n$ with $rn<2^{\omega(\sqrt{m}\log(m))}$,
\begin{eqnarray*}
[l^{n-m}{\perm}_m]\not\in \sigma_r(Ch_{n}(\mathbb{C}^{m^2+1})),
\end{eqnarray*}
then Valiant's Conjecture $\mathbf{VP\neq VNP}$ {\rm \cite{vali:79-3}} holds.
\end{theorem}

\begin{theorem}\label{con}{\rm \cite{DBLP:journals/eccc/GuptaKKS13,MR3343444}}
If for all but finite number of $n$, and for $\delta_1,\delta_2\sim\sqrt{n}$, for $r$,$\rho$ with $r\rho<2^{\omega({\sqrt{n}\tlog(n)})}$,
\begin{eqnarray*}
[{\perm}_n]\not\in \sigma_\rho(v_{\delta_1}(\sigma_r(v_{\delta_2}(\mathbb{P}^{n^2-1})))),
\end{eqnarray*}
then Valiant's conjecture $\mathbf{VP\neq VNP}$ {\rm \cite{vali:79-3}} holds.
\end{theorem}

 Theorems \ref{chowvnp} and \ref {con} motivated me to study the varieties $\s_r(Ch_d(V))$ and $\sigma_\rho(v_{\delta_1}(\sigma_r(v_{\delta_2}(\mathbb{P}V)))$.
 The results obtained here are not in the ranges needed to separate $\mathbf{VP}$ from $\mathbf{VNP}$. However, I introduce methods from representation theory for these problems. For the second problem this is the first time that equations for $\sigma_\rho(v_{\delta_1}(\sigma_r(v_{\delta_2}(\mathbb{P}V)))$ are approached from this perspective.

My results include
\begin{itemize}
\item Equations for $Ch_d(\mathbb{C}^d)$ ( Theorem \ref{rankchow}).
\item Equations for $\s_r(Ch_d(\mathbb{C}^{dr}))$ ( Theorem \ref{rankschow}).
\item A lower bound of $\underline{\mathbf{R}}_S(x_1\cdots x _{d})$ when $d$ is odd ( Theorem \ref{chowsrank}).
\item A lower bound on the size of depth 5 circuits that compute $\det_n$ and $\perm_n$ ( Theorem \ref{permcom})
\end{itemize}
\subsection{Flattenings and Koszul Young flattenings}

Equations for the secant varieties of Chow varieties are mostly unknown, and even for the secant varieties of Veronese varieties  very little is known.
One class of equations is  obtained from the so-called {\it flattenings} or {\it catalecticants}, which
date back to Sylvester:  for any $1\leq k\leq d$, there is an inclusion $F_{k,d-k}:S^dV\hookrightarrow S^kV\otimes S^{d-k}V$, called a polarization map. For any $P\in S^dV$, define the {\it $k$-th polarization} $P_{k,d-k}$ of $P$ to be $F_{k,d-k}(P)$. Then $P_{k,d-k}\in S^kV\otimes S^{d-k}V$ can be seen as a linear map $P_{k,d-k}: S^kV^*\rightarrow S^{d-k}V$. The image of $P_{k,d-k}$ is the space spanned by all $k$-th order partial derivatives of $P$, and is studied in the computer science literature under the name the {\it method
of partial derivatives} (see, e.g. \cite{MR2901512} and the references therein).
In coordinates, if $\{x_1,\dots,x_n\}$ is a basis of $V$,
then $\{\frac{\partial^k}{\partial  x_1^{i_1}\cdots \partial x_{n}^{i_{n}}}\}_{i_1+\cdots+i_{n}=k}$ is a basis of $S^kV^*$,
and $P_{k,d-k}(\frac{\partial^k}{\partial  x_1^{i_1}\cdots \partial x_{n}^{i_{n}}})=\frac{\partial^k P}{\partial  x_1^{i_1}\cdots \partial x_{n}^{i_{n}}}$.

If $[P]\in v_d(\mathbb{P}V)$, the rank
of $P_{k,d-k}$ is one, so the size $(r+1)$-minors of $P_{k,d-k}$ are in the ideal of  $I_{r+1}(\sigma_r(v_d(\mathbb{P} V)))$.
If $[P]\in Ch_d(V)$ with $\dim\ V\geq d$, then the rank of $P_{k,d-k}$ is $\binom dk$, so the size $r\binom dk +1$ minors are in the ideal of $\sigma_r(Ch_d(V))$.

Flattenings generalize to {\it Young flattenings}, see \cite{DBLP:journals/corr/EfremenkoLSW15,MR3427655,MR3081636} for a discussion of
Young flattenings and the state of the art. For $P\in S^dV$, the Koszul Young flattening is a linear map $P_{k,d-k}^{\wedge p}:S^kV^*\otimes \Lambda^pV\rightarrow S^{d-k+1}V\otimes \Lambda^{p+1}V$, it is defined by the composition of the following two maps
\begin{eqnarray*}
S^kV^*\otimes \Lambda^pV\rightarrow S^{d-k}V\otimes \Lambda^pV\rightarrow  S^{d-k-1}V\otimes\Lambda^{p+1}V,
\end{eqnarray*}
where the first map is defined by tensoring $P_{k,d-k}$ with the identity map  $Id_{\Lambda^p V}:\Lambda^pV\rightarrow \Lambda^pV$, and the second map $\wedge_{d-k,p}:S^{d-k}V\otimes\Lambda^pV\rightarrow S^{d-k-1}V\otimes\Lambda^{p+1}V$ is the algebraic exterior derivative, which is defined on monomials by:
\begin{eqnarray*}
  l_1\cdots l_{d-k}\otimes m_1\wedge m_2\cdots\wedge m_p\mapsto \sum_{s=1}^{d-k}l_1l_2\cdots \hat{l_s}\cdots l_{d-k}\otimes l_s\wedge m_1\wedge m_2\cdots\wedge m_p,
  \end{eqnarray*}
and then extended linearly.
 In the tensor setting, Koszul Young flattenings have led to the current best lower bound for the border rank of matrix multiplication in \cite{2016arXiv160108229L,v011a011}.

 Another Young flattening $P_{k,d-k[\ell]}: S^kV^*\ot S^{\ell}V\ra S^{d-k+\ell}V$ is obtained
 by tensoring $P_{k,d-k}$ with the identity map $Id_{S^{\ell}V}: S^{\ell}V\ra S^{\ell}V$, and projecting (symmetrizing)
 the image in $S^{d-k}V\ot S^{\ell}V$ to $S^{d-k+\ell}V$. It goes under the name \lq\lq the method of shifted partial
 derivatives\rq\rq\ in the computer science literature.
 By using the method of shifted partial derivatives , A. Gupta, P. Kamath, N. Kayal and R. Saptharishi \cite{gupta4} proved
 if $\delta_1,\delta_2 \sim \sqrt{n}$, $\dim\ V=n^2$ and $[\perm_n]\in \sigma_r(v_{\delta_1}(\mathbb{P}S^{\delta_2}V))$, then $r=2^{\Omega(\sqrt{n})}$.

By computing the Koszul Young flattenings of Chow Varieties and their secant varieties, one can obtain additional equations for these varieties.
\begin{theorem}\label{rankchow}
Let $V=\mathbb{C}^d$ with a basis $\{x_1,\dots,x_d\}$ and $P=x_1\cdots x_d$,  and let $1\leq k<d$, $1\leq p <d$.
The map
\begin{eqnarray*}
P_{k,d-k}^{\wedge p}:S^kV^*\otimes\Lambda^pV\rightarrow S^{d-k-1}V\otimes \Lambda^{p+1}V
\end{eqnarray*}
has rank\begin{eqnarray}
\mathbf{S(p,d,k)}&=&\sum_{s=\max\{0,p-k\}}^{\min\{p,d-k-1\}}\binom{d}{s}\binom{d-s}{d-k+p-2s}\binom{d-k+p-2s-1}{p-s}\\
&=&\frac{d!}{p!(d-p-1)!}\sum_{s=\max\{0,p-k\}}^{\min\{p,d-k-1\}}\frac{\binom{p}{s}\binom{d-1-p}{s+k-p}}{d-k+p-2s}.
\end{eqnarray}
In particular, when $\lceil\frac{d}{2}\rceil\leq k\leq d-3$ $(d\geq6)$, the $(\mathbf{S(p,d,k)}+1)\times (\mathbf{S(p,d,k)}+1)$ minors
of $P_{k,d-k}^{\wedge p}$ are in the ideal of $Ch_d(V)$.
\end{theorem}
\begin{remark}Proposition \ref{nontrivial} implies that the  equations
I obtain here are nontrivial when $\lceil\frac{d}{2}\rceil\leq k\leq d-3$.
\end{remark}
\begin{theorem}\label{rankschow}
Let $ V=\mathbb{C}^{rd}$ with a basis $\{x_1,\dots,x_{rd}\}$ and $P=x_1\cdots x_d+x_{d+1}\cdots x_{2d}+\cdots+x_{(r-1)d+1}\cdots x_{rd}$. Assume $1\leq k<d$, $1\leq p<d$, $r\geq2$.
Then the map
\begin{eqnarray*}
P_{k,d-k}^{\wedge p}:S^kV^*\otimes\Lambda^pV\rightarrow S^{d-k-1}V\otimes \Lambda^{p+1}V
\end{eqnarray*}
has rank
\begin{eqnarray}
\rank(P_{k,d-k}^{\wedge p})\leq r[\binom{d}{k}(\binom{dr}{p}-\binom{d}{p})+\mathbf{S(p,k,d)}].
\end{eqnarray}
In particular, when $d\geq2$, and $p=k=1$,
\begin{eqnarray*}
\rank (P_{1,d-1}^{\wedge1})\leq d^2r^2-r.
\end{eqnarray*}
Therefore the $(d^2r^2-r+1)\times(d^2r^2-r+1)$ minors of $P_{1,d-1}^{\wedge1}$ are in the ideal of
$\sigma_r(Ch_d(V))$.
\end{theorem}
\begin{remark}Theorem \ref{kyfl11}  implies that the  equations
I obtain here are nontrivial when $p=k=1$.
\end{remark}
Koszul Young flattenings can also be used to compute a lower bound for $\underline{\mathbf{R}}_S(x_1\cdots x _d)$.
By usual flattenings,  $\underline{\mathbf{R}}_S(x_1\cdots x _d)\geq \binom{d}{\lfloor\frac{d}{2}\rfloor}\sim\frac{2^d}{\sqrt{d}}$.
Ranestad and Schreyer \cite{MR2842085} showed $\mathbf{R}_S(x_1\cdots x _d)=2^{d-1}$. Therefore $\binom{d}{\lfloor\frac{d}{2}\rfloor}\leq\underline{\mathbf{R}}_S(x_1\cdots x _d)\leq\mathbf{R}_S(x_1\cdots x _d)\leq 2^{d-1}$.
 By computing Koszul Young flattenings, I get a lower bound of $\underline{\mathbf{R}}_S(x_1\cdots x _d)$ which is s better than $\binom{d}{\lfloor\frac{d}{2}\rfloor}$ when $d$ is odd, with an additional exponential term:
\begin{theorem}\label{chowsrank}
$\underline{\mathbf{R}}_S(x_1\cdots x _{2n+1})\geq\binom{2n+1}{n}(1+\frac{n^2}{(n+1)^2(2n-1)}).$
\end{theorem}
\begin{remark}
When $d=2n$, I conjecture that with the same method one can show $\underline{\mathbf{R}}_S(x_1\cdots x _{2n})\geq\binom{2n}{n}(1+\frac{C}{2n})$ for some constant $C$ and for $n$ big enough, and I verified many cases with a computer.
 Theorem \ref{rankchow} implies when $d=3$, $\underline{\mathbf{R}}_S(x_1x_2x_3)\geq 4$, so $\underline{\mathbf{R}}_S(x_1x_2x_3)=4>\binom{3}{1}$;
when $d=4$, $\binom{4}{2}<7\leq\underline{\mathbf{R}}_S(x_1x_2x_3x_4)\leq8 $;
when $d=5$, $\binom{5}{2}<14\leq\underline{\mathbf{R}}_S(x_1x_2x_3x_4x_5)\leq16 $; when $d=6$, $\binom{6}{3}<28\leq\underline{\mathbf{R}}_S(x_1x_2x_3x_4x_5)\leq32 $.
\end{remark}
I compute the flattening rank of a generic polynomial in $v_{\delta_1}(\sigma_r(v_{\delta_2}(\mathbb{P}V)))\subset\mathbb{P}S^{n}V$,
where $\dim\ V=n^2$ and $\delta_1,\delta_2\sim\sqrt{n}$, and I compare it to that of the permanent:
\begin{theorem}\label{permcom}
Let $\delta_1,\delta_2 \sim \sqrt{n}$, $\delta_1\delta_2=n $ and $\dim\ V=n^2$, when $\frac{2n-\sqrt{n}\tlog(r)}{\sqrt{n}\tlog(n)}=\omega(1)$, i.e. $r=2^{2\sqrt{n}-\tlog(n){\omega(1)}}$,
\begin{eqnarray*}
[{\perm}_n]\  \not\in\sigma_\rho( v_{\delta_1}(\sigma_r(v_{\delta_2}(\mathbb{P}V))))
\end{eqnarray*}
for $\rho<2^{\sqrt{n}\tlog(n)\omega(1)}$.
\end{theorem}
\begin{remark}
By using the model $\sigma_r(v_{\delta_1}(\mathbb{P}S^{\delta_2}V))$,$\delta_1,\delta_2 \sim \sqrt{n}$, where $\delta_1\delta_2=n $ and $\dim\ V=n^2$, and with the method of shifted partial derivatives, Gupta, Kamath, Kayal and Saptharishi \cite{gupta4} gave a lower bound $2^{\Omega(\sqrt{n})}$ for the permanent.  While by using the  model  $\sigma_\rho( v_{\delta_1}(\sigma_r(v_{\delta_2}(\mathbb{P}V))))$, where $\delta_1,\delta_2 \sim \sqrt{n}$, $\delta_1\delta_2=n $ and $\dim\ V=n^2$, and by the usual flattenings,  Theorem \ref{permcom} gives us a  lower bound $2^{\sqrt{n}\tlog(n)\omega(1)}$ for the permanent  when $r$ is relatively small (slightly smaller than $2^{2\sqrt{n}}$). However if $r$ is very big, we can not get much information by usual flattenings.
\end{remark}
\subsection{Organization}
In $\S\ref{disjoint}$ I define maps which I call disjoint linear maps, which are used to estimate the rank of flattenings and Koszul Young flattenings   of  polynomials in $\S\ref{ychow}$ and $\S\ref{cfvsv}$. In $\S\ref{ychow}$ I compute the  Koszul Young flattenings of $x_1x_2\cdots x_d$ and obtain equations for $Ch_d(\mathbb{C}^{d})$. Then I get a lower bound of symmetric border rank for $x_1x_2\cdots x_d\in S^d\mathbb{C}^{d}$ for $d$ odd. In $\S\ref{yschow}$ I compute the Koszul Young flattenings of $\sum_{i=1}^rx_{(i-1)d+1}x_{(i-1)d+2}\cdots x_{(i-1)d+d}\in S^d\mathbb{C}^{dr}$ and obtain equations for $\sigma_r(Ch_d(\mathbb{C}^{dr}))$. In $\S\ref{cfvsv}$ I compute the  flattening rank of a generic polynomial $[P]=[(x_1^{\delta_2}+\cdots x_r^{\delta_2})^{\delta_1}]\in v_{\delta_1}(\sigma_r(v_{\delta_2}(\mathbb{P}^{r-1})))$.
In $\S\ref{dpcom}$ I compare the flattening rank of a generic polynomial in $v_{\delta_1}(\sigma_r(v_{\delta_2}(\mathbb{P}^{n^2-1})))$ ($\delta_1,\delta_2\sim\sqrt{n}$) with that of the  $\perm_n$, and get a new complexity lower bound for the permanent as long as $r$ is relatively small.
\subsection{Acknowledgement}
I thank my advisor J.M. Landsberg for discussing all the details throughout this article. I thank Y. Qi for discussing the second part of this article. I thank Iarrobino for providing references for Theorem \ref{classic}.
Part of this work was done while the author was visiting the Simons Institute for the Theory of Computing, UC Berkeley for the Algorithms and Complexity in Algebraic Geometry program in 2014, I thank Simons Institute for providing a good research environment.
\section{ Koszul Young Flattenings of Chow varieties and their secant varieties}\label{KyChow}
\subsection{Disjoint linear maps}\label{disjoint}
\begin{definition}
Let $V$ and $W$ to be two finite dimensional complex vector spaces and let $f:V\rightarrow W$ be a linear map.
 Let $V_1,V_2,\dots V_m$ be subspaces of $V$ such that $V=\bigoplus_{i=1}^mV_i$. The map $f$ is called a disjoint map with respect to the decomposition $V={\bigoplus}_{i=1}^mV_i$ if $f(V)=\bigoplus_{i=1}^m f(V_i)$.
\end{definition}

Note that if $f:V\rightarrow W$ is a disjoint linear map with respect to the decomposition  $V=\bigoplus_{i=1}^mV_i$,
then $\rank(f)={\sum}_{i=1}^m \rank(f|_{V_i})$.

\subsection{Koszul Young flattenings of Chow varieties}\label{ychow}
Let $V=\mathbb{C}^d$ with  basis $\{x_1,\dots,x_d\}$. Let $P=x_1\cdots x_d$ and write $[d]=\{1,\dots,d\}$. The following proposition is standard:
\begin{proposition}\label{partial}
Let  $1\leq k<d$, then the image of $$P_{k,d-k}:S^kV^*\rightarrow S^{d-k}V$$
is $(S^{d-k}V){{\rm reg}}:={\rm span}\{x_{i_1}x_{i_2}\cdots x_{i_{d-k}}\}_{1\leq i_1<i_2<\cdots <i_{d-k}\leq d}$.
\end{proposition}
By Proposition \ref{partial},
\begin{corollary}\label{wedge}
Let  $1\leq k<d$, $1\leq p<d$,
the image of
\begin{eqnarray*}
P_{k,d-k}^{\wedge p}:S^kV^*\otimes\Lambda^pV\rightarrow S^{d-k-1}V\otimes \Lambda^{p+1}V
\end{eqnarray*}
is the image of $(S^{d-k}V)_{{\rm reg}}\otimes\Lambda^pV$ under the map
\begin{eqnarray*}
\wedge_{d-k,p}:S^{d-k}V\otimes\Lambda^pV\rightarrow S^{d-k-1}V\otimes\Lambda^{p+1}V.
\end{eqnarray*}
\end{corollary}
Since the map $\wedge_{d-k,p}|_{(S^{d-k}V)_{{\rm reg}}\otimes\Lambda^pV}$ preserves weights, it is helpful for us to decompose $(S^{d-k}V)_{{\rm reg}}\otimes\Lambda^pV$ into a direct sum of weight spaces.
\begin{lemma}\label{decom}
Let $W_{k_1,..,k_s;j_1,j_2,\cdots,j_{d-k+p-2s}}$ be the span of
$\{x_{k_1}\cdots x_{k_s}x_{m_1}\cdots x_{m_{d-k-s}}\otimes x_{k_1}\wedge\cdots x_{k_s}\wedge x_{n_1}\wedge\cdots \wedge x_{n_{p-s}}\}_{\{m_1,\dots,m_{d-k-s},n_1,\dots,n_{p-s}\}=\{j_1,j_2,\dots,j_{d-k+p-2s}\}}.$
Then
\begin{eqnarray}\label{decomp3}
&&(S^{d-k}V)_{{\rm reg}}\otimes\Lambda^pV=\\\nonumber
&&\bigoplus_{\{k_1,..,k_s\}\subset [d]\ \max\{0,p-k\}\leq s\leq \min\{p,d-k\}}\bigoplus_{\{j_1,j_2,\dots,j_{d-k+p-2s}\}\subset [d]- \{k_1,...,k_s\}}W_{k_1,...,k_s;j_1,j_2,\dots,j_{d-k+p-2s}}.
\end{eqnarray}
Moreover
$\wedge_{d-k,p}|_{(S^{d-k}V)_{{\rm reg}}\otimes\Lambda^pV}$ is a disjoint map with respect to this decomposition.
\end{lemma}
\begin{example}
Let $d=3$ and $k=p=1$, then
\begin{eqnarray*}
(S^{2}V)_{{\rm reg}}\otimes V=\bigoplus_{1\leq i,j\leq3\ i\neq j}W_{i;j}\bigoplus W_{;1,2,3}.
\end{eqnarray*}
Where $W_{i;j}=\spa\{x_ix_j\otimes x_i\}$, and $W_{;1,2,3}=\spa\{x_1x_2\otimes x_3, x_1x_3\otimes x_2, x_2x_3\otimes x_1\}$.
\end{example}
For the symmetric group $\mathfrak{S}_n$ and any partition $\lambda$ with order $n$,  let $[\lambda]$  denote the irreducible $\mathfrak{S}_n$-module corresponding to the partition $\lambda$.
\begin{lemma}\label{keylemma}
Let $Q={\rm span}\{y_1,...,y_{u+v}\}$ and let
\begin{eqnarray*}
A_{u,v}Q={\rm span}\{y_{m_1}\cdots y_{m_u}\otimes y_{n_1}\wedge\cdots\wedge y_{n_v}\}_{\{m_1,\dots,m_u,n_1,\dots,n_v\}=[u+v]}.
\end{eqnarray*}
Then the rank of $\wedge_{u,v}|_{A_{u,v}Q}:A_{u,v}Q\rightarrow A_{u-1,v+1}Q$ is $\binom{u+v-1}{v}$.
Moreover \begin{eqnarray*}
\rank(\wedge_{d-k,p}|_{W_{k_1,...,k_s;j_1,j_2,\dots,j_{d-k+p-2s}}})=\rank(\wedge_{d-k-s,p-s}|_{A_{d-k-s,p-s}Q})=\binom{d-k+p-2s-1}{p-s}.
\end{eqnarray*}
\end{lemma}
\begin{proof}Recall that $[u+v-1,1]$ is the standard representation of $\mathfrak{S}_{u+v}$ and $\Lambda^s[u+v-1,1]=[u+v-s,1^s]$ ( see\cite{FH} Exercise 4.6).
As a $\mathfrak{S}_{u+v}$-module, $A_{u,v}Q=\Lambda^u([u+v-1,1]+\mathbb{C})=\Lambda^u{[u+v-1,1]}+\Lambda^{u-1}[u+v-1,1]=[u,1^v]\oplus [u+1,1^{v-1}]$ and
$A_{u-1,v+1}Q=[u-1,1^{v+1}]\oplus [u,1^v]$, since $\wedge_{u,v}|_{A_{u,v}Q}$ is a
$\mathfrak{S}_{u+v}$-module map, by Schur's lemma, image$(\wedge_{u,v}|_{A_{u,v}Q})=[u,1^v]$,
whose dimension is $\binom{u+v-1}{v}$.
We can see $\wedge_{d-k,p}|_{W_{k_1,...,k_s;j_1,j_2,\dots,j_{d-k+p-2s}}}$ is essentially the same map as $\wedge_{d-k-s,p-s}|_{A_{d-k-s,p-s}Q}$,
so \begin{eqnarray*}
\rank(\wedge_{d-k,p}|_{W_{k_1,..,k_s;j_1,j_2,\cdots,j_{d-k+p-2s}}})=\rank(\wedge_{d-k-s,p-s}|_{A_{d-k-s,p-s}Q})=\binom{d-k+p-2s-1}{p-s}.
\end{eqnarray*}
\end{proof}
\begin{theorem}{\rm (\cite{MR0480472,MR515043,MR1735271})}\label{classic}
 For a generic polynomial $P\in S^dV$, and for any $1\leq k<d$, the flattening $P_{k,d-k}: S^kV^*\rightarrow S^{d-k}V$
 is of maximal rank, i.e. $\min \{\binom{k+\dim V-1}{k},\binom{d-k+\dim V-1}{d-k}\}$.
\end{theorem}
 For any partition $\lambda$, let $S_\lambda V$  denote the irreducible $GL(V)$-module corresponding to the partition $\lambda$.
\begin{proposition}\label{nontrivial}
For a generic polynomial $P\in S^dV$,
the rank  of
\begin{eqnarray*}
P_{k,d-k}^{\wedge p}:S^kV^*\otimes\Lambda^pV\rightarrow S^{d-k-1}V\otimes \Lambda^{p+1}V
\end{eqnarray*}
is $\dim\ S_{(d-k,1^p)}V=\frac{d}{d-k+p}\binom{2d-k-1}{d}\binom{d-1}{p}$.
When $\lceil\frac{d}{2}\rceil\leq k\leq d-3$,
rank$((x_1\cdots x_d)_{k,d-k}^{\wedge p})$
is less than that of a generic polynomial in $S^dV$.
\end{proposition}
\begin{proof}
First, assume that $P$ is a generic polynomial in  $S^dV$, and
$\lceil\frac{d}{2}\rceil\leq k$, by Theorem \ref{classic},
$$\rank(P_{k,d-k}^{\wedge p})=\rank(\wedge_{d-k,p}),$$
where \begin{eqnarray*}
\wedge_{d-k,p}:S^{d-k}V\otimes\Lambda^pV\rightarrow S^{d-k-1}V\otimes\Lambda^{p+1}V.
\end{eqnarray*}
By Pieri's rule, $S^{d-k}V\otimes\Lambda^pV=S_{(d-k,1^p)}V\oplus S_{(d-k+1,1^{p-1})}V$
and $S^{d-k-1}V\otimes\Lambda^{p+1}V=S_{(d-k,1^p)}V\oplus S_{(d-k-1,1^{p+1})}V$.
Therefore by Schur's lemma, ${\rm image}(\wedge_{d-k,p})=S_{(d-k,1^p)}V$, and by Hook-content formula, $\rank(\wedge_{d-k,p})=\dim\ S_{(d-k,1^p)}V=\frac{d}{d-k+p}\binom{2d-k-1}{d}\binom{d-1}{p}$.

Second assume  $P=x_1\cdots x_d\in S^dV$ and $\lceil\frac{d}{2}\rceil\leq k\leq d-3$,
by Corollary \ref{wedge},
the image of
\begin{eqnarray*}
P_{k,d-k}^{\wedge p}:S^kV^*\otimes\Lambda^pV\rightarrow S^{d-k-1}V\otimes \Lambda^{p+1}V
\end{eqnarray*}
is the image of $(S^{d-k}V)_{{\rm reg}}\otimes\Lambda^pV$ under the map
\begin{eqnarray*}
\wedge_{d-k,p}:S^{d-k}V\otimes\Lambda^pV\rightarrow S^{d-k-1}V\otimes\Lambda^{p+1}V.
\end{eqnarray*}
 Easy to see that $\wedge_{d-k,p}(x_1^{d-k}\otimes( x_2\wedge\cdots\wedge x_{p+1}))=(d-k)x_1^{d-k-1}\otimes(x_1\wedge x_2\wedge\cdots\wedge x_{p+1})\in {\rm image}(\wedge_{d-k,p})$,
but $x_1^{d-k-1}\otimes(x_1\wedge x_2\wedge\cdots\wedge x_{p+1})$ is not contained in
the image of $\wedge_{d-k,p}|_{(S^{d-k}V)_{{\rm reg}}\otimes\Lambda^pV}$, because $x_1$ appears at most twice in the image and $d-k\geq3$.
The result follows.
\end{proof}
\begin{proof}[Proof of Theorem \ref{rankchow}]
First, by Corollary \ref{wedge}, we only need to compute the rank of
\begin{eqnarray*}
\wedge_{d-k,p}|_{(S^{d-k}V)_{{\rm reg}}\otimes\Lambda^pV}:(S^{d-k}V)_{{\rm reg}}\otimes\Lambda^pV\rightarrow S^{d-k-1}V\otimes\Lambda^{p+1}V.
\end{eqnarray*}
By Lemma \ref{keylemma}, $\rank(\wedge_{d-k,p}|_{W_{k_1,...,k_s;j_1,j_2,\dots,j_{d-k+p-2s}}})$ depends only on $s$.
Consider the decomposition of ${(S^{d-k}V)_{{\rm reg}}\otimes\Lambda^pV}$ in \eqref{decomp3}, for any given $s$, the number of subspaces $W_{k_1,...,k_s;j_1,j_2,\dots,j_{d-k+p-2s}}$ is $\binom{d}{s}\binom{d-s}{d-k+p-2s}$. Therefore, by Lemma \ref{decom} and Lemma \ref{keylemma}.
\begin{eqnarray}
\nonumber\mathbf{S(p,d,k)}&=&\sum_{s=\max\{0,p-k\}}^{\min\{p,d-k\}}\binom{d}{s}\binom{d-s}{d-k+p-2s}\rank(\wedge_{d-k,p}|_{W_{k_1,...,k_s;j_1,j_2,\dots,j_{d-k+p-2s}}})\\
&=&\sum_{s=\max\{0,p-k\}}^{\min\{p,d-k-1\}}\binom{d}{s}\binom{d-s}{d-k+p-2s}\binom{d-k+p-2s-1}{p-s}\\
&=&\frac{d!}{p!(d-p-1)!}\sum_{s=\max\{0,p-k\}}^{\min\{p,d-k-1\}}\frac{\binom{p}{s}\binom{d-1-p}{s+k-p}}{d-k+p-2s}.
\end{eqnarray}
Second, by Proposition \ref{nontrivial}, when $\lceil\frac{d}{2}\rceil\leq k\leq d-3$,
the $(\mathbf{S(p,d,k)}+1) \times (\mathbf{S(p,d,k)}+1)$ minors are in the ideal of $Ch_d(V)$.
\end{proof}

\begin{lemma}\label{YFveronese}
Let $P=l^d$ for some linear form $l\in V$, then $\rank(P_{k,d-k}^{\wedge p})=\binom{d-1}{p}$.
\end{lemma}
\begin{proof}
Let $P=l^d$, the image of
\begin{eqnarray*}
P_{k,d-k}^{\wedge p}:S^kV^*\otimes\Lambda^pV\rightarrow S^{d-k-1}V\otimes \Lambda^{p+1}V
\end{eqnarray*}
is span$\{l^{d-k-1}\otimes (l\wedge x_{i_1}\wedge\cdots\wedge x_{i_p})\}$. We can assume $l=x_1$, then
$\rank(P_{k,d-k}^{\wedge p})=\binom{d-1}{p}$.
\end{proof}
\begin{proof}[Proof of Theorem \ref{chowsrank}]
When $d=2n+1$ and $k=p=n$,
\begin{equation}
\mathbf{S(n,2n+1,n)}=\frac{(2n+1)!}{(n!)^2}\sum_{s=0}^{n}\frac{\binom{n}{s}^2}{1+2s}.
\end{equation}
and  by Lemma \ref{YFveronese}, $\underline{\mathbf{R}}_S(x_1\cdots x _{2n})\geq\frac{\mathbf{S(n,2n+1,n)}}{\binom{2n}{n}}$.

Let$$\mathbf{\mathbf{A}}=\sum_{s=0}^{n}\frac{\binom{n}{s}^2}{1+2s}.$$
Then\begin{eqnarray*}
2\mathbf{\mathbf{A}}&=&\sum_{s=0}^{n}\frac{\binom{n}{s}^2}{1+2s}+\sum_{s=0}^{n}\frac{\binom{n}{s}^2}{1+2n-2s}\\
&=&\sum_{s=0}^{n}\binom{n}{s}^2(\frac{1}{1+2s}+\frac{1}{1+2n-2s})\\
&=&\sum_{s=0}^{n}\binom{n}{s}^2\frac{2+2n}{(1+2s)(1+2n-2s)}.\\
\end{eqnarray*}

Notice that $$(1+2s)(1+2n-2s)\leq(1+n)^2.$$

So\begin{eqnarray*}
\mathbf{\mathbf{A}}&=&\sum_{s=0}^{n}\binom{n}{s}^2\frac{1+n}{(1+2s)(1+2n-2s)}\\
&=&\sum_{s=0}^{n}\binom{n}{s}^2(\frac{1+n}{(1+2s)(1+2n-2s)}-\frac{1}{1+n}+\frac{1}{1+n})\\
&=&\sum_{s=0}^{n}\binom{n}{s}^2(\frac{n^2-4ns+4s^2}{(n+1)(1+2s)(1+2n-2s)}+\frac{1}{1+n})\\
&\geq&\sum_{s=0}^{n}\binom{n}{s}^2(\frac{n^2-4ns+4s^2}{(n+1)^3}+\frac{1}{1+n})\\
&=&\sum_{s=0}^{n}\binom{n}{s}^2\frac{n^2-4ns+4s^2}{(n+1)^3}+\binom{2n}{n}\frac{1}{1+n}.\\
\end{eqnarray*}
Since $$\sum_{s=0}^{n}\binom{n}{s}^2=\binom{2n}{n},$$
$$\sum_{s=0}^{n}s\binom{n}{s}^2=n\binom{2n-1}{n},$$
and$$\sum_{s=0}^{n}s(s-1)\binom{n}{s}^2=n(n-1)\binom{2n-2}{n},$$
we have $$\sum_{s=0}^{n}s^2\binom{n}{s}^2=n(n-1)\binom{2n-2}{n}+n\binom{2n-1}{n}.$$

This implies \begin{eqnarray*}
\mathbf{\mathbf{A}}&\geq&\sum_{s=0}^{n}\binom{n}{s}^2\frac{n^2-4ns+4s^2}{(n+1)^3}+\binom{2n}{n}\frac{1}{1+n}\\
&=&\frac{n^2\binom{2n}{n}+4n(n-1)\binom{2n-2}{n}+4n\binom{2n-1}{n}-4n^2\binom{2n-1}{n}}{(n+1)^3}+\binom{2n}{n}\frac{1}{1+n}\\
&=&\frac{\binom{2n}{n}}{(n+1)^3}[n^2+4n(n-1)\frac{n(n-1)}{2n(2n-1)}+(4n-4n^2)\frac{n}{2n}]+\binom{2n}{n}\frac{1}{1+n}\\
&=&\frac{\binom{2n}{n}}{(n+1)^3}(\frac{2n(n-1)^2}{2n-1}-n^2+2n)+\binom{2n}{n}\frac{1}{1+n}\\
&=&\frac{\binom{2n}{n}}{(n+1)^3}\frac{2n(n-1)^2-(n^2-2n)(2n-1)}{2n-1}+\binom{2n}{n}\frac{1}{1+n}\\\\
&=&\frac{\binom{2n}{n}}{(n+1)^3}\frac{n^2}{2n-1}+\binom{2n}{n}\frac{1}{1+n}.\\
\end{eqnarray*}

Therefore
\begin{eqnarray*}
\frac{\mathbf{S(n,2n+1,n)}}{\binom{2n}{n}\binom{2n+1}{n}}&=&\frac{(2n+1)!}{(n!)^2}\frac{\mathbf{\mathbf{A}}}{\binom{2n}{n}\binom{2n+1}{n}}\\
&\geq&\frac{(2n+1)!}{(n!)^2}\frac{\frac{\binom{2n}{n}}{(n+1)^3}\frac{n^2}{2n-1}+\binom{2n}{n}\frac{1}{1+n}}{\binom{2n}{n}\binom{2n+1}{n}}\\
&=&\frac{(2n+1)!}{(n!)^2}\frac{1}{\binom{2n+1}{n}}(\frac{1}{1+n}+\frac{n^2}{(n+1)^3(2n-1)})\\
&=&(n+1)(\frac{1}{1+n}+\frac{n^2}{(n+1)^3(2n-1)})\\
&=&1+\frac{n^2}{(n+1)^2(2n-1)}.\\
\end{eqnarray*}
Therefore $\underline{\mathbf{R}}_S(x_1\cdots x _{2n})\geq\frac{\mathbf{S(n,2n+1,n)}}{\binom{2n}{n}}\geq\binom{2n+1}{n}(1+\frac{n^2}{(n+1)^2(2n-1)}).$
\end{proof}
\subsection{Koszul Young flattenings of secant varieties of Chow varieties}\label{yschow}
I study equations for secant varieties of Chow varieties from the perspective of Koszul Young flattenings.

\begin{proposition}
 Let $ V=\mathbb{C}^{dr}$ with a basis $\{x_1,\dots,x_{dr}\}$. Let $P=x_1\cdots x_d+x_{d+1}\cdots x_{2d}+\cdots+x_{(r-1)d+1}\cdots x_{rd}$. For any $k\leq\lfloor\frac{d}{2}\rfloor$,
 the map $P_{k,d-k}:S^kV^*\rightarrow S^{d-k}V$ has rank
$$\rank(P_{k,d-k})=r\binom{d}{k}.$$
Therefore the $(r\binom{d}{k}+1)\times (r\binom{d}{k}+1)$ minors of the linear map $P_{k,d-k}$ are in the ideal of $\sigma_r(Ch_d(V))$.
In particular, when $k=2$ and $d\geq4$, $r\binom{d}{2}<\binom{rd+1}{2}$, so we obtain equations for $\sigma_r(Ch_d(V))$
of degree $r\binom{d}{2}+1$.
\end{proposition}
\begin{theorem}\label{kyfl11}
 Let $ V=\mathbb{C}^{n}$ with a basis $\{x_1,\dots,x_{n}\}$, for a generic polynomial $P\in S^dV$, the rank of $ P_{1,d-1}^{\wedge 1}:V^*\otimes V\rightarrow S^{d-1}V\otimes \Lambda^{2}V$ is $n^2-1$.
\end{theorem}
\begin{proof}
First, let $P\in S^dV$, then
$$P_{1,d-1}^{\wedge 1}(\sum_{i=1}^nx_i^*\otimes x_i)=\sum_{i=1}^n\sum_{j=1}^n P_{x_ix_j}\otimes (x_j\wedge x_i)=0$$
so $\rank( P_{1,d-1}^{\wedge 1})\leq n^2-1$.

Second, let $P=x_1^d+\cdots+x_n^d+x_1^{d-1}(x_2+\cdots+x_n)$,
then
$P_{1,d-1}^{\wedge 1}(x_i^*\otimes x_k)$ is
\begin{eqnarray*}
\begin{cases}
d(d-1)x_i^{d-2}\otimes (x_i\wedge x_k)+d(d-1)x_1^{d-2}\otimes (x_1\wedge x_k)\  i=2,\cdots,n,\\
d(d-1)x_1^{d-2}\otimes (x_1\wedge x_k)+\sum_{j=2}^d[(d-1)x_1^{d-2}\otimes (x_j\wedge x_k)+(d-1)(d-2)x_1^{d-3}x_j\otimes (x_1\wedge x_k)\ i=1.
\end{cases}
\end{eqnarray*}
If we choose the $n^2-1$ dimensional subspace $A_n$ of $V^*\otimes V$ with basis indexed by $x_i^*\otimes x_k$ for $k$ and $i$ are not both 1,
and choose the $n^2-1$ dimensional subspace $B_n$ of $S^{d-1}V\otimes \Lambda^{2}V$ with basis  indexed by $x_i^{d-2}\otimes (x_i\wedge x_k)$ for $i=2,\cdots,n$ and $k\neq i$,
 $x_1^{d-2}\otimes (x_1\wedge x_k)$ for $k\neq1$, and $x_1^{d-3}x_2\otimes (x_1\wedge x_k)$ for $k\neq1$,
 then by projection to $B_n$, we can rewrite $P_{1,d-1}^{\wedge 1}(x_i^*\otimes x_k)$ as
\begin{eqnarray*}
\begin{cases}
d(d-1)x_i^{d-2}\otimes (x_i\wedge x_k)+d(d-1)x_1^{d-2}\otimes (x_1\wedge x_k)\  i=2,\cdots,n,\\
d(d-1)x_1^{d-2}\otimes (x_1\wedge x_k)+(d-1)(d-2)x_1^{d-3}x_2 \otimes (x_1\wedge x_k)]\ i=1,\ k\neq1.
\end{cases}
\end{eqnarray*}
It is easy to see they are linearly independent, so  $\rank( P_{1,d-1}^{\wedge 1})= n^2-1$, the result follows.
\end{proof}
While we can not obtain equations of $\sigma_r(Ch_3(\mathbb{C}^{3r}))$ just by usual flattenings, we
can  obtain equations for $\sigma_r(Ch_3\mathbb{C}^{3r})$ by Koszul Young flattenings in Theorem \ref{rankschow}.

\begin{proof}[Proof of Theorem \ref{rankschow}]
First, let $V_i=\spa\{x_{(i-1)d+1},\dots,x_{(i-1)d+d}\}$, for $i=1,2,\dots,r$.
Then $V=\oplus_{i=1}^rV_i$. The image of the map
\begin{eqnarray*}
P_{k,d-k}^{\wedge p}:S^kV^*\otimes\Lambda^pV\rightarrow S^{d-k-1}V\otimes \Lambda^{p+1}V
\end{eqnarray*}
is the image of the map
\begin{eqnarray*}
\wedge_{d-k,p}:\bigoplus_{i=1}^r(S^{d-k}V_i)_{{\rm reg}}\otimes\Lambda^pV\rightarrow S^{d-k-1}V\otimes\Lambda^{p+1}V.
\end{eqnarray*}
 Write $\Lambda^pV=\Lambda^pV_i\oplus W_i$, where $W_i$ is the complement of $\Lambda^pV_i$ with respect to the basis $\{x_{i_1}\wedge\cdots\wedge x_{i_p}\}_{1< i_1< i_2<\cdots< i_p\leq dr}$.
 Rewrite the map as
 \begin{eqnarray*}
 \wedge_{d-k,p}:\bigoplus_{i=1}^r(S^{d-k}V_i)_{{\rm reg}}\otimes(\Lambda^pV_i\oplus W_i)\rightarrow S^{d-k-1}V\otimes\Lambda^{p+1}V.
 \end{eqnarray*}
 Then
 \begin{eqnarray*}
\rank(P_{k,d-k}^{\wedge p})&\leq&\sum_{i=1}^r\rank(\wedge_{d-k,p}|_{(S^{d-k}V_i)_{{\rm reg}}\otimes\Lambda^pV_i})+\rank(\wedge_{d-k,p}|_{(S^{d-k}V_i)_{{\rm reg}}\otimes W_i})\\
&\leq& r[\mathbf{S(p,k,d)}+\binom{d}{k}(\binom{dr}{p}-\binom{d}{p})].
\end{eqnarray*}
Second, when $d\geq2$, $r\geq2$ and $p=k=1$,
$$ \rank (P_{1,d-1}^{\wedge1})\leq d^2r^2-r.$$
By Theorem \ref{kyfl11}, the $(d^2r^2-r+1)\times(d^2r^2-r+1)$ minors of $P_{1,d-1}^{\wedge1}$ are in the ideal of
$\sigma_r(Ch_d(V))$.
\end{proof}
\section{Flattenings of Veronese reembeddings of secant varieties of Veronese varieties}\label{fvsv}
\subsection{Flattenings of Veronese reembeddings of secant varieties of Veronese varieties}\label{cfvsv}
Let $\{x_1,\dots,x_r\}$ be a basis of $V$, and $\{y_1,\dots,y_r\}$ be the dual basis of $V^*$. Let $P=(x_1^{\delta_2}+\cdots+ x_r^{\delta_2})^{\delta_1}\in S^dV$, where $d=\delta_1\delta_2$, note that $[P]$ is a generic element of $v_{\delta_1}(\sigma_r(v_{\delta_2}(\mathbb{P}^{r-1})))$. The goal is to compute the rank of its $(k, d-k)$-flattening $P_{k,d-k}:S^kV^*\rightarrow S^{d-k}V,$ where $ 1\leq k<d$.
\begin{definition}
Let $y_1^{\alpha_1}\cdots y_r^{\alpha_r}\in S^kV^*$,
the {\it support} of $P_{k,d-k}(y_1^{\alpha_1}\cdots y_r^{\alpha_r})$ is the set of all monomials
appearing in $P_{k,d-k}(y_1^{\alpha_1}\cdots y_r^{\alpha_r})$.
\end{definition}
\begin{example}
Consider $\delta_1=\delta_2=r=k=2$, then $P=(x_1^2+x_2^2)^2$,
$P_{2,2}(y_1^2)=12x_1^2+x_2^2$, the support of $ P_{2,2}(y_1^2)$ is $\{x_1^2,x_2^2\}$,
Similarly the support of $P_{2,2}(y_2^2)$ is $\{x_1^2,x_2^2\}$, and the support of $ P_{2,2}(y_1y_2)$ is $\{x_1x_2\}$.
\end{example}
\begin{proposition}\label{class}Let $y_1^{\alpha_1}\cdots y_r^{\alpha_r}$ and $y_1^{\eta_1}\cdots y_r^{\eta_r}\in S^kV^*$, then
$P_{k,d-k}(y_1^{\alpha_1}\cdots y_r^{\alpha_r})$ and $P_{k,d-k}(y_1^{\eta_1}\cdots y_r^{\eta_r})$ have
the same support in $S^{d-k}V$ if and only if  $\alpha_i-\eta_i=n_i\delta_2$ for some integers $n_i$ and for each $i=1,\dots,r$.
\end{proposition}
\begin{proof}
\begin{eqnarray*}
P&=&(x_1^{\delta_2}+\cdots+ x_r^{\delta_2})^{\delta_1}\\
&=&\sum_{t_1+\cdots+t_r=\delta_1}\binom{\delta_1}{t_1,\cdots,t_r}x_1^{t_1\delta_2}x_2^{t_2\delta_2}\cdots x_r^{t_r\delta_2}.\\
\end{eqnarray*}
Therefore
\begin{eqnarray*}
P_{k,d-k}(y_1^{\alpha_1}\cdots y_r^{\alpha_r})=\sum_{t_1+\cdots+t_r=\delta_1}C(t_1,\dots,t_r;\alpha_1,\dots,\alpha_r)x_1^{t_1\delta_2-\alpha_1}x_2^{t_2\delta_2-\alpha_2}\cdots x_r^{t_r\delta_2-\alpha_r}.
\end{eqnarray*}
Where $C(t_1,\dots,t_r;\alpha_1,\dots,\alpha_r)$ are coefficients depending on ($t_1,\dots,t_r$) and ($\alpha_1,\dots,\alpha_r$).
\begin{eqnarray*}
P_{k,d-k}(y_1^{\eta_1}\cdots y_r^{\eta_r})=\sum_{t_1+\cdots+t_r=\delta_1}C(t_1,\dots,t_r;\eta_1,\dots,\eta_r)x_1^{t_1\delta_2-\eta_1}x_2^{t_2\delta_2-\eta_2}\cdots x_r^{t_r\delta_2-\eta_r}
\end{eqnarray*}
Therefore $P_{k,d-k}(y_1^{\alpha_1}\cdots y_r^{\alpha_r})$ and $P_{k,d-k}(y_1^{\eta_1}\cdots y_r^{\eta_r})$ have
the same support in $S^{d-k}V$ if and only if  $\alpha_i-\eta_i=n_i\delta_2$ for some integers $n_i$ and for each $i=1,\dots,r$.
\end{proof}

\begin{definition}  For  $\alpha_1+\cdots+\alpha_r=k$, define the subspace $A[\alpha_1,\dots,\alpha_r]\subset S^kV^*$ by
\begin{eqnarray*}
&&A[\alpha_1,\dots,\alpha_r]=\\
&&{\rm span}\{y_1^{\eta_1}\cdots y_r^{\eta_r}\in S^kV^*|\alpha_i-\eta_i=n_i\delta_2\ {\rm for\ some\ integer\ } n_i{\rm\ and\ for\ each\ } i=1,\cdots,r\}.
\end{eqnarray*}
Define $B[\alpha_1,\dots,\alpha_r]$ to be the subspace of $S^{d-k}V$ spanned by the support of $P_{k,d-k}(y_1^{\alpha_1}\cdots y_r^{\alpha_r})$.
\end{definition}

With the notation above,  $S^kV^*$ can be decomposed as a direct sum of subspaces $A[\alpha_1,\cdots,\alpha_r]$. By Proposition \ref{class},
weights of vectors in distinct $B[\alpha_1,\dots,\alpha_r]$ are distinct, therefore $P_{k,d-k}$ is a disjoint map with respect to this decomposition.

\begin{definition}
For each subspace $A[\alpha_1,\dots,\alpha_r]$, write $\alpha_i=\beta_i+s_i\delta_2$, where $0\leq\beta_i<\delta_2$. It is easy to see each $\beta_i$ and $\sum_{i=1}^rs_i$ are the same for each basis vector.
Define two  functions
\begin{equation}
A(\alpha_1,\dots,\alpha_r)=\sum_{i=1}^rs_i=\sum_{i=1}^r\lfloor\frac{\alpha_i}{\delta_2}\rfloor.
\end{equation}
and \begin{equation}
B(\alpha_1,\dots,\alpha_r)=\delta_1-\sum_{i=1}^r\lceil\frac{\alpha_i}{\delta_2}\rceil.
\end{equation}
\end{definition}

One can check $0\leq A(\alpha_1,\cdots,\alpha_r)\leq \lfloor\frac{k}{\delta_2}\rfloor$ and
\begin{equation}\label{AB}
B(\alpha_1,\dots,\alpha_r)=\delta_1-A(\alpha_1,\dots,\alpha_r)-\#\{\beta_i>0\}
\end{equation}
\begin{proposition}
Consider $$P_{k,d-k}|_{A[\alpha_1,\dots,\alpha_r]}:A[\alpha_1,\dots,\alpha_r]\rightarrow B[\alpha_1,\dots,\alpha_r],$$
then $\dim\ A[\alpha_1,\dots,\alpha_r]=\binom{A(\alpha_1,\dots,\alpha_r)+r-1}{A(\alpha_1,\dots,\alpha_r)}$ and $\dim\  B[\alpha_1,\dots,\alpha_r]=\binom{B(\alpha_1,\dots,\alpha_r)+r-1}{B(\alpha_1,\dots,\alpha_r)}.$
\end{proposition}
\begin{proof}
Write $\alpha_i=\beta_i+s_i\delta_2$, where $0\leq\beta_i<\delta_2$.
Then
\begin{eqnarray*}
&&A[\alpha_1,\dots,\alpha_r]=\\
&&{\rm span}\{y_1^{\eta_1}\cdots y_r^{\eta_r}\in S^kV^*|\eta_i=n_i\delta_2+\beta_i{\rm\ for}\ n_i\geq0{\rm\ such\ that\ }n_1+\cdots+n_r=A(\alpha_1,\dots,\alpha_r)\}.
\end{eqnarray*}
So for each basis vector $y_1^{\eta_1}\cdots y_r^{\eta_r}\in A[\alpha_1,\dots,\alpha_r]$,
$y_1^{\eta_1}\cdots y_r^{\eta_r}=(y_1^{\beta_1}\cdots y_r^{\beta_r})(y_1^{n_1}\cdots y_r^{n_r})^{\delta_2}$
with $n_1+\cdots+n_r=A(\alpha_1,\dots,\alpha_r)$ and $n_i\geq0$.

Therefore  dim $A[\alpha_1,\dots,\alpha_r]=\binom{A(\alpha_1,\dots,\alpha_r)+r-1}{A(\alpha_1,\dots,\alpha_r)}$ .\\
Similarly let
\begin{eqnarray*}
\theta_i=\begin{cases}
\delta_2-\beta_i\ \beta_i\neq0\\
0\ \ \ \ \ \ \ \ \beta_i=0.
\end{cases}
\end{eqnarray*}
Let $x_1^{\zeta_1}\cdots x_r^{\zeta_r}\in B[\alpha_1,\dots,\alpha_r]$, then $\zeta_i=m_i\delta_2+\theta_i$ for some nonnegative integer $m_i$,
and
\begin{eqnarray*}
\sum_{i=1}^r(m_i\delta_2+\theta_i)&=&\delta_1\delta_2-k\\
&=&\delta_1\delta_2-\sum_{i=1}^r\alpha_i\\
&=&\delta_1\delta_2-\sum_{i=1}^r(\beta_i+s_i\delta_2)\\
&=&\delta_1\delta_2-A(\alpha_1,\dots,\alpha_r)\delta_2-\sum_{i=1}^r\beta_i\\
&=&(\delta_1-A(\alpha_1,\dots,\alpha_r)-\#\{\beta_i>0\})\delta_2+\sum_{i=1}^r\theta_i.
\end{eqnarray*}
So
\begin{eqnarray*}
m_1+\cdots+m_r=\delta_1-A(\alpha_1,\dots,\alpha_r)-\#\{\beta_i>0\}=B(\alpha_1,\dots,\alpha_r).
\end{eqnarray*}
Therefore
\begin{eqnarray*}
\begin{aligned}
&B[\alpha_1,\dots,\alpha_r]=\\
&{\rm span}\{x_1^{\zeta_1}\cdots x_r^{\zeta_r}\in S^{d-k}V|\zeta_i=m_i\delta_2+\theta_i{\rm\ for\ }m_i\geq0\ {\rm \ such\ that\ }m_1+\cdots+m_r=B(\alpha_1,\dots,\alpha_r)\}.
\end{aligned}
\end{eqnarray*}
So for each basis vector $x_1^{\zeta_1}\cdots x_r^{\zeta_r}\in B[\alpha_1,\dots,\alpha_r]$,
$x_1^{\zeta_1}\cdots x_r^{\zeta_r}=(x_1^{\theta_1}\cdots x_r^{\theta_r})(x_1^{m_1}\cdots x_r^{m_r})^{\delta_2}$
with $m_1+\cdots+m_r=B(\alpha_1,\dots,\alpha_r)$ and $m_i\geq0$.
Therefore dim $B[\alpha_1,\dots,\alpha_r]=\binom{B(\alpha_1,\dots,\alpha_r)+r-1}{B(\alpha_1,\dots,\alpha_r)}$.
\end{proof}
\begin{remark} Spaces $A[\alpha_1,\dots,\alpha_r]$ and $B[\alpha_1,\dots,\alpha_r]$ are essentially determined by $(\beta_1,\dots,\beta_r)$ with $0\leq \beta_i<\delta_2$ $(i=1,\dots,r)$.
\end{remark}

\begin{definition}
Define $\NUM(A,B)$ to be the number of subspaces $A[\alpha_1,\dots,\alpha_r]$ of $S^kV^*$ such that $A(\alpha_1,\dots,\alpha_r)=A$ and $B(\alpha_1,\dots,\alpha_r)=B$.
\end{definition}
One can decompose
\begin{eqnarray}
\label{decomp}S^kV^*=\bigoplus_{A=0}^{\lfloor\frac{k}{\delta_2}\rfloor}\bigoplus _{B\leq\delta_1-A}W(A,B)^{\oplus \NUM(A,B)}
\end{eqnarray}
such that $W(A,B)$ is a copy of subspace $A[\alpha_1,\dots,\alpha_r]$ in $S^kV^*$ with $A(\alpha_1,\dots,\alpha_r)=A$ and $B(\alpha_1,\dots,\alpha_r)=B$.
Furthermore, by Proposition \ref{class}, $P_{k,d-k}$ is a disjoint map with respect to this decomposition.
\begin{example}
Consider $\delta_1=\delta_2=r=k=2$, and the decomposition of $S^2V^*$.
One can see $A[2,0]=A[0,2]=\spa\{y_1^2,y_2^2\}$ with $A(2,0)=A(0,2)=0$ and $B(2,0)=B(0,2)=1$;
$A[1,1]=\spa\{y_1y_2\}$ with  $A(1,1)=B(1,1)=0$.
Therefore $\NUM(0,0)=\NUM(0,1)=1$, $ \NUM(A,B)=0$ otherwise,
and $$S^2V^*=W(0,1)\oplus W(0,0).$$
\end{example}
\begin{proposition}\label{NUMAB}
We have
$$\NUM(A,B)=\#\{(\beta_1,\dots,\beta_r)|\beta_1+\cdots +\beta_r=k-A\delta_2 ,\#\{\beta_i>0\}=\delta_1-B-A\ {\rm and}\ 0\leq\beta_i<\delta_2\ {\rm for}\ i=1,\dots,r\}.$$
Moreover the following bounds hold for the rank of $P_{k,d-k}$,
\begin{eqnarray}
\NUM(0,0)\leq \rank(P_{k,d-k})\leq\sum_{A=0}^{\lfloor\frac{k}{\delta_2}\rfloor}\sum_{B=0}^{\delta_1-A}\min\{\binom{A+r-1}{A},\binom{B+r-1}{B}\}\NUM(A,B).
\end{eqnarray}
\end{proposition}
\begin{proof}
First let $A[\alpha_1,\dots,\alpha_r]$ be a subspace of $S^kV^*$ such that $A(\alpha_1,\dots,\alpha_r)=A$ and $B(\alpha_1,\dots,\alpha_r)=B$.
Write $\alpha_i=\beta_i+n_i\delta_2$, where $0\leq\beta_i<\delta_2$, then
\begin{eqnarray*}
k&=&\sum_{i=1}^rn_i\delta_2+\beta_i\\
&=&A\delta_2-\sum_{i=1}^r\beta_i
\end{eqnarray*}
Therefore $\beta_1+\cdots +\beta_r=k-A\delta_2$. By \eqref{AB}, $\#\{\beta_i>0\}=\delta_1-B-A\ and\ 0\leq\beta_i<\delta_2$.
On the other hand, since $(\beta_1,\dots,\beta_r)$ determine the class $A[\alpha_1,\dots,\alpha_r]$ uniquely,
$\NUM(A,B)=\#\{(\beta_1,\dots,\beta_r)|\beta_1+\cdots +\beta_r=k-A\delta_2 ,\#\{\beta_i>0\})=\delta_1-B-A\ {\rm and}\ 0\leq\beta_i<\delta_2\ {\rm for}\ i=1,\dots,r\}.$

Second, by \eqref{decomp} one can decompose $S^kV^*=\bigoplus_{A=0}^{\lfloor\frac{k}{\delta_2}\rfloor}\bigoplus _{B\leq\delta_1-A}W(A,B)^{\oplus \NUM(A,B)}$
such that $W(A,B)$ is a copy of subspace $A[\alpha_1,\dots,\alpha_r]$ in $S^kV^*$ with $A(\alpha_1,\dots,\alpha_r)=A$ and $B(\alpha_1,\dots,\alpha_r)=B$, and $P_{k,d-k}$ is a disjoint map with respect to the decomposition.
Since when $B<0$, $ P_{k,d-k}|_{W(A,B)}$ is the zero map,
\begin{eqnarray*}
\rank(P_{k,d-k})&=&\sum_{A=0}^{\lfloor\frac{k}{\delta_2}\rfloor}\sum_{B=0}^{\delta_1-A}\rank(P_{k,d-k}|W(A,B))\NUM(A,B)\\
&\leq&\sum_{A=0}^{\lfloor\frac{k}{\delta_2}\rfloor}\sum_{B=0}^{\delta_1-A}\min\{\binom{A+r-1}{A},\binom{B+r-1}{B}\}\NUM(A,B).
\end{eqnarray*}
On the other hand, consider subspaces
$A[\alpha_1,\dots,\alpha_r]$ with $A(\alpha_1,\dots,\alpha_r)=0$ and $B(\alpha_1,\dots,\alpha_r)$\\$=0$, the map
$P_{k,d-k}|_{A[\alpha_1,\dots,\alpha_r]}:A[\alpha_1,\dots,\alpha_r]\rightarrow B[\alpha_1,\dots,\alpha_r]$ is just a nonzero linear map from
a one-dimensional space to another one-dimensional space. Therefore
\begin{eqnarray*}
\rank(P_{k,d-k})&=&\sum_{A=0}^{\lfloor\frac{k}{\delta_2}\rfloor}\sum_{B=0}^{\delta_1-A}\rank(P_{k,d-k}|W(A,B))\NUM(A,B)\\
&\geq&\NUM(0,0).
\end{eqnarray*}
\end{proof}
\begin{corollary}\label{bounds}
Let $\delta_1\delta_2=d$ with $\delta_1,\delta_2\sim \sqrt{d}$  and  let $r\geq2\delta_1$, then
\begin{eqnarray}
\binom{r}{\delta_1}\leq \rank([(x_1^{\delta_2}+\cdots+ x_r^{\delta_2})^{\delta_1}]_{\lfloor\frac{d}{2}\rfloor,\lceil\frac{d}{2}\rceil})\leq \delta_1\binom{r}{\delta_1}\binom{\lfloor\frac{d}{2}\rfloor-1}{\delta_1-1}.
\end{eqnarray}
\end{corollary}
\begin{proof}
Write $k=\lfloor\frac{d}{2}\rfloor$ and $P=(x_1^{\delta_2}+\cdots+ x_r^{\delta_2})^{\delta_1}$, by Proposition \ref{NUMAB},
\begin{eqnarray*}
\rank(P_{k,d-k})&\geq&\NUM(0,0)\\
&=&\#\{(\beta_1,\dots,\beta_r)|\beta_1+\cdots +\beta_r=k ,\#\{\beta_i>0\})=\delta_1\ \\
&&{\rm and}\ 0\leq\beta_i<\delta_2\ {\rm for}\ i=1,\dots,r\}\\
&=&\binom{r}{\delta_1}\#\{(\beta_1,\dots,\beta_{\delta_1})|\beta_1+\cdots +\beta_{\delta_1}=k ,\ 0<\beta_i<\delta_2\ {\rm for}\ i=1,\dots,\delta_1\}\\
&\geq&\binom{r}{\delta_1}.
\end{eqnarray*}\\
To prove the second inequality, notice that each monomial of $P=(x_{1}^{\delta_2}+\cdots+x_{r}^{\delta_2})^{\delta_1} $ is of the form $x_{i_1}^{\delta_2}\cdots x_{i_{\delta_1}}^{\delta_2}$ for $1\leq i_1\leq\cdots i_r\leq r$. Let $y_1^{\alpha_1}\cdots y_r^{\alpha_r}\in S^k{V^*}$, if the number of positive $\alpha_i$  is greater than $\delta_1$, then each monomial of $P$  vanishes  when we take partial derivatives, i.e. $x_{i_1}^{\delta_2}\cdots x_{i_{\delta_1}}^{\delta_2}(y_1^{\alpha_1}\cdots y_r^{\alpha_r})=0$, so $P_{k,d-k}(y_1^{\alpha_1}\cdots y_r^{\alpha_r})=0$ under this condition. Therefore
\begin{eqnarray*}
\rank(P_{k,d-k})\leq \#\{(\alpha_1,\dots,\alpha_r)|\ {\rm where}\ \alpha_1+\cdots+\alpha_r=k\ {\rm and}\
 \#\{\alpha_i>0\}\leq\delta_1\}.
\end{eqnarray*}
 So
 \begin{eqnarray*}
\rank(P_{k,d-k})&\leq&\sum_{s=1}^{\delta_1}\binom{r}{s}\binom{k-1}{s-1}\\
&\leq&\delta_1\binom{r}{\delta_1}\binom{k-1}{\delta_1-1}\\
\end{eqnarray*}
\end{proof}

\begin{remark}
While we can get better bounds for $\rank(P_{k,d-k})$, it is always the case that $r$ is much bigger than $d=\delta_1\delta_2$,
therefore the term $\binom{r}{\delta_1}$ dominates.
\end{remark}
\subsection{Comparison with the permanent}\label{dpcom}
Let $\{x_1,...,x_{n^2}\}$ be a basis of $V$ and $r\geq n$, and let $[P]=[(l_{1}^{\delta_2}+\cdots+l_{r}^{\delta_2})^{\delta_1}]\in v_{\delta_1}(\sigma_r(v_{\delta_2}(\mathbb{P}V)))\subset \mathbb{P}S^nV$ be generic, where $\delta_1,\delta_2 \sim \sqrt{n}$, $\delta_1\delta_2=n$. I compute the rank of the flattening $P_{\lfloor\frac{n}{2}\rfloor,\lceil\frac{n}{2}\rceil}:S^{\lfloor\frac{n}{2}\rfloor}V^*\rightarrow S^{\lceil\frac{n}{2}\rceil}V$ and compare it with $\rank((\perm_n)_{\lfloor\frac{n}{2}\rfloor,\lceil\frac{n}{2}\rceil})$.

Classical results show for any $0\leq k\leq\lfloor\frac{n}{2}\rfloor$,
$\rank(\perm_n)_{k,n-k}=\binom{n}{k}^2$.
\begin{lemma}\label{proj}
Let $\tilde{A},\tilde{B}$ be two complex vector spaces, let $A$ be a subspace of $\tilde{A}$ and $B$ be a subspace of $\tilde{B}$.
If $\widetilde{T}\in \tilde{A}\otimes\tilde{B}$ and $T\in A\otimes B$ is the linear projection of $\tilde{T}$, then
$R(\tilde{T})\geq R(T)$.
\end{lemma}
\begin{corollary}\label{bounds1}
Assume that $\delta_1,\delta_2 \sim \sqrt{n}$ and $\delta_1\delta_2=n$, then
$$\rank([(l_{1}^{\delta_2}+\cdots+l_{r}^{\delta_2})^{\delta_1}]_{\lfloor\frac{n}{2}\rfloor,\lceil\frac{n}{2}\rceil})\leq \delta_1\binom{r}{\delta_1}\binom{\lfloor\frac{n}{2}\rfloor-1}{\delta_1-1}\leq(r\lfloor\frac{n}{2}\rfloor)^{\delta_1}$$.
\end{corollary}
\begin{proof}
Write $k=\lfloor\frac{n}{2}\rfloor$ and $P=(l_{1}^{\delta_2}+\cdots+l_{r}^{\delta_2})^{\delta_1}$. Assume that ${l_1,...,l_r}$ are linearly independent and let $ W=$span$\{l_1,..., l_r\}$, in  Corollary \ref{bounds} we already computed the rank of the new flattening ${\tilde{P}_{k,n-k}:S^kW^*\rightarrow S^{n-k}W}$. By Lemma \ref{proj} $\rank(P_{k,n-k})\leq \rank(\tilde{P}_{k,n-k})$.\\
By Corollary \ref{bounds},
 \begin{eqnarray*}
 \rank(P_{k,n-k})&\leq &\rank(\tilde{P}_{k,n-k})\\
 &\leq&\delta_1\binom{r}{\delta_1}\binom{k-1}{\delta_1-1}\\
&\leq&(rk)^{\delta_1}
\end{eqnarray*}
\end{proof}

\begin{proof}[Proof of Theorem \ref{permcom}]
Let $k=\lfloor\frac{n}{2}\rfloor$, and $2n-\sqrt{n}\tlog(r)=\sqrt{n}\tlog(n)\omega(1)$, by Corollary \ref{bounds1},
\begin{eqnarray*}
  \frac{\rank((\perm_n)_{k,n-k})}{ \rank(P_{k,n-k})}&\geq &\frac{\binom{n}{\lfloor\frac{n}{2}\rfloor}^2}{(rk)^{\delta_1}}\\
 &\sim &\frac{2^{2n}}{n2^{\sqrt{n}\tlog(rn/2)}}\\
 &=&2^{2n-\sqrt{n}\tlog(r)-\sqrt{n}\tlog(n/2)-\tlog(n)}\\
  &=&2^{\sqrt{n}\tlog(n)\omega(1)-\sqrt{n}\tlog(n/2)-\tlog(n)}\\
 &=& 2^{\sqrt{n}\tlog(n)\omega(1)}.
\end{eqnarray*}
\end{proof}

\bibliographystyle{amsplain}
\bibliography{Lmatrix}
\end{document}